\documentclass[12pt]{article}
\usepackage{amsmath,amssymb}
\usepackage{amsthm}
\usepackage{cite}

\newtheorem{theorem}{Theorem}

\newtheorem{lemma}{Lemma}

\def\PC{\mathcal{P}}
\def\MC{\mathcal{M}}

\def\B{\mathbf{B}}

\def\R{\mathbf{R}}

\def\1{\mathbf{1}}

\def\al{\alpha}
\def\be{\beta}
\def\pa{\partial}
\def\ep{\epsilon}
\def\de{\delta}
\def\ga{\gamma}

\newcommand{\la}{\lambda}
\newcommand{\si}{\sigma}

\newcommand{\Om}{\Omega}

\begin{document}
	
\title{Regularity and Sensitivity for McKean-Vlasov Type SPDEs Generated by Stable-like Processes
	%\thanks{preprint}
}
\author{
	Vassili N. Kolokoltsov\thanks{Department of Statistics, University of Warwick,
		Coventry CV4 7AL UK, Email: v.kolokoltsov@warwick.ac.uk, and associate member of Faculty of Applied Mathematics and Control Processes of Saint-Petersburg State University, and Institute of Informatics Problems of the Federal Research Center ``Computer Science and Control'' of RAS
}
	and Marianna Troeva\thanks{Research Institute of Mathematics, North-Eastern Federal University,
		58 Belinskogo str., Yakutsk 677000 Russia, Email: troeva@mail.ru
	}}
	\maketitle

\begin{abstract}
		In this paper we study the sensitivity of nonlinear stochastic differential equations of McKean-Vlasov type generated by stable-like processes. By using the method of stochastic characteristics, we transfer these equations to the non-stochastic equations with random coefficients thus making it possible to use the results obtained for nonlinear PDE of McKean-Vlasov type generated by stable-like processes in the previous works. The motivation for studying sensitivity of nonlinear McKean-Vlasov SPDEs arises naturally from the analysis of the mean-field games with common noise.
\end{abstract}

\noindent
{\bf Mathematics Subject Classification (2010)}: {60H15, 60J60, 91A06, 91A15,\\ 49L20, 82C22}

\section{Introduction}

In this paper we shall study the well-posedness and sensitivity of the following stochastic partial differential equation of the McKean-Vlasov type generated by stable-like processes, 
\begin{equation}
\label{eqlimMacVlaSPDE}
d(f, \mu_t)=(L(\mu_t)f +\frac12
(\sigma_{com}\sigma_{com}^T\nabla,\nabla)f, \mu_t) \, dt +(\si_{com}\nabla f, \mu_t) \, dW_t.
\end{equation}

This equation is written in the weak form meaning that it should hold for all $f \in C^2(\R^d)$.
Here $x\in \R^d$, $W_t$ is $d'$-dimensional standard Brownian motion,
$\sigma_{com}$ is a constant $d\times d'$ matrix,
\begin{equation}
\label{eqlimMacVlaSPDE1}
L(\mu_t)f (x)= (b(x, \mu_t),\nabla)f (x) - a(x)|\Delta|^{\alpha/2}f(x),
\end{equation}
is the generator for the stable-like
processes in $\mathbf{R}^{d}$ with stability index $\alpha\in (1,2)$,
$\mu_t \in \mathbf{{\mathcal{M}}}\left(\mathbf{R}^{d} \right)$ (the
set of bounded positive Borel measures on $\mathbf{R}^{d}$), the drift
 $b(x,\mu)$  and the scale coefficient $a(x)$ are continuous functions.

By the usual rule $Y\circ dX=Y dX +\frac12 dY dX$, equation (\ref{eqlimMacVlaSPDE})
rewrites in a more transparent Stratonovich form as
\begin{equation}
\label{eqlimMacVlaSPDEst}
d(f, \mu_t)=(L(\mu_t) f, \mu_t) \, dt +
(\si_{com} \nabla f, \mu_t) \, \circ dW_t.
\end{equation}

Recall that the fractional Laplacian can be expressed via the integral
operator, i.e.,
\begin{equation}\label{fractdiffer}
|\Delta|^{\alpha/2}f(x)=C_{\al}\int_{\mathbf{R}^{d}}\left(f(x+y)-f(x)-\frac{(\nabla f(x),y)}{1+|y|^2}\right)\frac{dy}{|y|^{d+\alpha}},
\end{equation}
with a certain constant $C_{\al}$.

The motivation for studying sensitivity of McKean-Vlasov SPDEs (\ref{eqlimMacVlaSPDE}) and the notation $\sigma_{com}$
arise naturally from the analysis of the mean-field games with common noise, in which positions of $N$ agents are governed by the system of SDEs
\begin{equation}
	\label{eqstartSDEN}
	dX_t^i=b(X_t^i, \mu_t^N, u_t^i)\, dt +
	\si_{com}dW_t + a^{1/\al}(X_t^i) dY^i_t,
\end{equation}
where all $X_t^i$ belong to $\R^d$, $W_t$ is a multidimensional standard Brownian motion
referred to as the common noise and $Y^i_t$ are independent symmetric L\'evy processes with the index $\al$.   
The parameters $u_t^i\in U\subset \R^m$ are controls available to the players, trying to minimize their payoffs. McKean-Vlasov equation is a forward component of the forward-backward systems of equations expressing of the mean-field game consistency problem, the backward component being the Hamilton-Jacobi equation. 

There is an extensive literature on properties of McKean-Vlasov SPDEs  (see e.g. \cite{KurXi, DaVa95, HuNu13, HuMalhCa06, KolTro15, CarmDel18, CrMcMur, CardDelLasLi15, Nour13, BenFrYam15, BasHilKol17, BardiPri15}  
 and references therein), based on the diffusion proccesses.  The well-posedness of the McKean-Vlasov SPDE was shown in \cite{KurXi} in the class of $L_2$-functions, and for 
 measures in \cite{DaVa95}, though under an additional monotonicity assumption. Our paper is the fisrt one where the McKean-Vlasov type SPDEs for the stable-like underlying Markov processes are analysed.

The sensitivity analysis for the nonlinear McKean-Vlasov diffusions and nonlinear stable-like processes stressing precise estimates of growth of the solutions and their derivatives with respect to the initial data, under rather general assumptions on the coefficients was studied in \cite{KolTro17, Kol1, KolTro16}. The exact estimates become important when treating the extension of these equations having random coefficient, since the noise is usually assumed to be unbounded.

Our paper is organized as follows. We first summarize the method of stochastic characteristics, see \cite{Kunita, KoTyu03, KolTro17}
in a simplified version used here. It is our main tool that allows us to turn stochastic  McKean-Vlasov equation into a non-stochastic equation with random coefficients.
Then we prove the well-posedness of equation \eqref{eqlimMacVlaSPDEst}. Finally
we prove our main results on the smooth sensitivity of this equation with respect to initial data.

The following basic notations will be used:

$C^n(\mathbf{R}^d)$ is the Banach space of $n$ times continuously differentiable and bounded functions $f$ on $\mathbf{R}^d$  such that each derivative up to and including order $n$ is bounded, equipped with norm $\|f\|_{C^n}$ which is the sum of the suprema of the magnitudes of all mixed derivatives up to and including order $n$.

$C_{\infty}(\R^d)$ is a Banach space of bounded continuous functions $f:\R^d\to \R$ with $ \lim_{x\to \infty}f(x)=0$, equipped with sup-norm.

$C_{\infty}^n(\mathbf{R}^d)$ is a closed subspace of $C^n(\mathbf{R}^d)$ with $f$ and all its derivatives up to and including order $n$ belonging to $C_{\infty}(\mathbf{R}^d)$.

If $\MC$ is a closed subset of a Banach space $\B$, then

$C([0,T], \MC)$ is a metric space of continuous functions $t \rightarrow \mu_t \in \MC$
with distance $\|\eta-\xi\|_{C([0,T], \MC)}=\sup_{t\in [0,T]} \|\eta_t-\xi_t\|_\B$. An element from $C([0,T], \MC)$ is written as $\{\mu_.\}=\{\mu_t, t\in[0,T]\}$.

%%%%%%%%%%%%%%%%%%из учебника

$\MC (\R^d)$ is a Banach space of finite signed Borel measures on  $\R^d$.

$\MC^+(\R^d)$ and  $\PC (\R^d)$ the subsets of $\MC (\R^d)$ of positive and positive normalised (probability) measures, respectively.

Let $\MC_{<\la}(\R^d)$ (resp. $\MC_{\le \la}(\R^d)$
 or $\MC_{\la}(\R^d)$) and $\MC_{<\la}^+(\R^d)$ (resp. $\MC_{\le \la}^+(\R^d)$ or $\MC_{\la}(\R^d)$)
denote the parts of these sets
containing measures of the norm less than $\la$ (resp. not exceeding $\la$ or equal $\la$).

Let $C^{k\times k}(\R^{2d})$ denote the subspace of $C(\R^{2d})$ consisting of functions $f$ such that the partial derivatives $\pa^{\al+\be}f/\pa x^{\al} \pa y^{\be}$ with multi-index $\al,\be$, $|\al|\le k, |\be|\le k$, are well defined and belong to $C(\R^{2d})$. Supremum of the norms of these derivatives provide the natural norm for this space.

%
%$C^{2,k \times k}_{weak} (\MC^+_{\le \la}(\R^d))$ is the subspace of $C^{1,1}_{weak} (\MC^+_{\le \la}(\R^d))$ consisting of functionals $F(\mu)$ such that $\de ^2F(Y)/\de Y(x)\de Y(z)$ exists for all $x,z$ and belongs to $C^{k\times k}(\R^{2d})$  uniformly for $Y\in \MC^+_{\le \la}(\R^d)$.

For a function $F$ on $\MC^+_{\le \la}(\R^d)$ or $\MC_{\le \la}(\R^d)$ the variational derivative is defined as the directional derivative of $F(\mu)$ in the direction $\delta_x$:
\[
\frac{\de F(\mu)}{\de \mu(x)}=\frac{d}{dh}|_{h=0}F(\mu +h\de_x).
\]
The higher derivatives
$\delta^l F(\mu)/\de \mu(x_1)...\de \mu(x_l)$ are defined inductively.

%%%%%%%%%%%%%%%%%%%%%%%%%%%%%% из учебника

Let $C^k(\MC_{\le \la}(\R^d))$ denote the space of functionals such that the $k$th order variational derivatives are well defined
and represent continuous functions of all variables with measures considered in their weak topology.

Let $C^{k,l}(\MC_{\le \la}(\R^d))$ denote the subspace of functionals $F$ from $C^k(\MC_{\la}(\R^d))$ such that
$\delta^m F(\mu)/\de \mu(.)...\de \mu(.) \in C^l(\R^d)$ for all $m\le k$.

Let $C^{2,k\times k}(\MC_{\le \la}(\R^d))$ be the space of functionals such that their second order variational derivatives
are continuous as functions of all variables and belong to
$C^{k\times k}(\R^{2d})$ as functions of the spatial variable; the norm of this space is
\[
\|F\|_{C^{2,k\times k}(\MC_{\le \la}(\R^d))}
=\sup_{\mu\in \MC_{\la}(\R^d)}\left \|\frac{\de^2 F}{\de \mu (.) \de \mu(.)}\right\|_{C^{k \times k}(\R^{2d})}.
\]

$(f,\mu)=\int f(x)\mu(dx)$ denotes the usual pairing of functions and measures on $\R^d$.

\section{Stochastic characteristics for commuting groups}

Let $L_t(\mu)$ be a family of operators defined on a dense subspace $D$ of a Banach space $B$
and depending on $\mu \in B^*$ as a parameter. Let $\Om=(\Om_1, \cdots,  \Om_k)$ be a vector of commuting
 linear operators $D\to B$, which generate commuting strongly continuous groups $T_j=\exp\{t\Om_j\}$ in $B$,
 and hence also weakly continuous groups  $T_j^*=\exp\{t\Om^*_j\}$ in $B^*$.
Let $W_t=(W^1, \cdots , W^k_t)$ be a $k$-dimensional standard Brownian motion.
We are interested in the weak Stratonovich SDE in $B^*$:
\begin{equation}
\label{eqgenchar1}
d(f,\mu_t)=(L_t(\mu_t) f, \mu_t) dt +(\Om f, \mu_t) \circ dW_t=(L_t(\mu_t) f, \mu_t) dt +\sum_j(\Om_jf, \mu_t) \circ dW^j_t.
\end{equation}

\begin{lemma}
\label{stochchara}
Under the change of the unknown function $\mu_t$ to
\[
\zeta_t=  \exp\{-\Om^* W_t\} \mu_t= \exp\{-\sum_j \Om^*_j W^j_t\} \mu_t,
\]
equation \eqref{eqgenchar1} transfers to a non stochastic PDE with random coefficients:
\begin{equation}
\label{eqgenchar2}
\frac{d}{dt}(f,\zeta_t)=(\tilde L_t (\zeta_t)f,\zeta_t)= (L^{dress}_t(\exp\{\Om^* W_t\} \zeta_t) f, \zeta_t),
\end{equation}
with
\begin{equation}
\label{eqgenchar3}
L^{dress}_t(\mu)f= \exp\{\Om W_t\}L_t(\mu) \exp\{-\Om W_t\}f, 
\quad \tilde L_t (\zeta_t)=L^{dress}_t(\exp\{\Om^* W_t\} \zeta_t). 
\end{equation}
\end{lemma}

\begin{proof}
Since the Stratonovich differential is subject to the usual rules of calculus, we have
\[
 d(f, \zeta_t)=d(  \exp\{-\Om W_t\}f, \mu_t)=-(\Om\exp\{-\Om W_t\}f, \mu_t) \circ dW_t+d(g, \mu_t)|_{g=\exp\{-\Om W_t\}f}
 \]
 \[
 =(L_t(\mu_t) \exp\{-\Om W_t\}f, \mu_t) dt= (\exp\{\Om W_t\}L_t (\mu_t)\exp\{-\Om W_t\}f, \zeta_t),
 \]
 yielding  \eqref{eqgenchar2}. An alternative proof can be given by first rewriting equation  \eqref{eqgenchar2}
 in the Ito form and then perform the transformation using Ito's formula.
\end{proof}

As we are mostly interested in the sensitivity, let us formulate the corresponding abstract result,
which is a direct consequence of Lemma \ref{stochchara}.
Let us concentrate on the case when $B=C_{\infty}(\R^d)$ and $B^*=\MC(\R^d)$, where the derivatives
can be written in terms of the variational derivatives.

\begin{lemma}
\label{stochcharb}
Let equation \eqref{eqgenchar2} be well posed and its solutions $\zeta_t$ depend smoothly on the initial condition
in the sense that the variational derivatives $\de \zeta_t/ \de \zeta_0 (x)$ and $\de^2 \zeta/ \de \zeta_0 (x) \de \zeta_0 (y)$
exist as signed measures and are continuous bounded functions of $x$ and $y$.
Then the solutions  $\mu_t=  \exp\{\Om^* W_t\} \zeta_t$ to equation \eqref{eqlimMacVlaSPDEst} also depend  smoothly
on the initial condition $\mu_0=\zeta_0$  and the variational derivatives are given by the formulas
\begin{equation}
\label{eqgenchar3}
\left(f, \frac{\de \mu_t}{\de \mu_0(x)}\right) = \frac{\de}{\de \mu_0(x)}(f, \mu_t) =\frac{\de}{\de \zeta_0(x)}(\exp\{\Om W_t\} f, \zeta_t),
\end{equation}
\begin{equation}
\label{eqgenchar4}
\left(f, \frac{\de^2 \mu_t}{\de \mu_0(x) \de \mu_0(y)}\right) = \frac{\de^2}{\de \zeta_0(x)\de \zeta_0(y)}(\exp\{\Om W_t\} f, \zeta_t).
\end{equation}
\end{lemma}

\section{A well posedness result} %Common Noise with Constant Correlations}
%\section{SENSITIVITY FOR STOCHASTIC MCKEAN-VLASOV EQUATIONS}

In this section we shall study the well-posedness of equations
\eqref{eqlimMacVlaSPDE} or \eqref{eqlimMacVlaSPDEst}.

Equation \eqref{eqlimMacVlaSPDEst} is an example of equation \eqref{eqgenchar1} with
\[
\Omega v(x) = (\si_{com},\nabla)v(x)=\{\Omega^j v(x)\}= \{ \sum_k \si_{com}^{jk} \nabla_k v(x)\}.
\]
Operator $\Omega$ has the dual operator
\[
\Omega' v(x) = -\Omega v(x)=(-\si_{com},\nabla)v(x),
\]
and generate $d'$ semigroups
\[
T_j(t)v(x)=v(x+\si_{com}^{j .}t)=v(x_j+\si_{com}^{jk}t),
\]
solving the Cauchy problems for the equations
\[
\frac{\pa v}{\pa t}(t,x) =\sum_j \si_{com}^{jk} \frac{\pa v}{\pa x_j}(t,x).
\]

These semigroups commute and define the action $T(t_1, \cdots , t_{d'})$ of $\R^{d'}$ on $(\R^d)$ by the formula
\[
T(t_1, \cdots , t_{d'}): v(x) \mapsto v(x+ \sum_j \si_{com}^{j.} t_j).
\]

According to Lemma \ref{stochchara}, equation  \eqref{eqlimMacVlaSPDEst} rewrites in terms
of the measures $\zeta_t=T^*(-W_t)\mu_t$ as equation \eqref{eqgenchar2}
\begin{equation}
\label{eqLdres1}%\label{eqgenchar2}
\frac{d}{dt}(f,\zeta_t)=(\tilde L_t (\zeta_t)f,\zeta_t)= (L^{dress}_t( \zeta_t(x-\si_{com}W_t)) f, \zeta_t),
\end{equation}
with
\[
L^{dress}_t(\mu)f(x)= \exp\{\Om W_t\}L(\mu) \exp\{-\Om W_t\}f(x)
\]
\begin{equation}
\label{eqLdres}
=(b(x+\si_{com} W_t,\mu) , \nabla) f(x)- a(x+\si_{com} W_t)|\Delta|^{\alpha/2}f(x)
\end{equation}
and 
\[
\tilde L_t(\zeta_t)f(x)=L^{dress}_t(\zeta_t(x-\si_{com}W_t))f(x)
\]
 \begin{equation}
 \label{eqLdres1_}
=(\tilde b(W_t,x,\zeta_t), \nabla) f(x)- \tilde a(W_t,x)|\Delta|^{\alpha/2}f(x),
\end{equation}
where
\[
\tilde b(W_t,x,\zeta_t)= b(x+\si_{com} W_t, \zeta_t(x-\si_{com}W_t)), \quad \tilde a (W_t,x)=a(x+\si_{com} W_t).
\]

We will use several regularity assumptions for the function $b(x,\mu)$ depending on $x$ and $\mu\in \MC(\R^d)$.
For all practical purposes the dependence of $b$ on $\mu$ is expressed in terms of certain integral functionals, i.e. is of the type
 \begin{equation}
\label{eqfuncb}
b(x,\mu)=b(x, I_1, \cdots, I_{\rho}),
\end{equation}
\[
I_m=\int_{\R^{dl_m}} B_m(x, y_1, \cdots , y_{l_m} )\mu(dy_1) \cdots \mu(dy_{l_m}),
\]       
with symmetric in $y$s functions $B_m$,
for which all variational derivatives belong to the same class:
 \begin{equation}
\label{eqfuncbder}
\frac{\de b(x,\mu)}{\de \mu(z)} =\sum_m \frac{\pa b}{\pa I_m}(x, I_1, \cdots, I_{\rho}) 
l_m \int_{\R^{d(l_m-1)}} B_m(x, z, y_2, \cdots , y_{l_m} )\mu(dy_2) \cdots \mu(dy_{l_m}), 
\end{equation}
and hence all regularity conditions can be rewritten in terms of the usual smoothness of functions $B_m$. 
 
Let us introduce the following conditions:

(C1) Function $a(x)\in C_\infty^2\left(\mathbf{R}^d\right)$ and satisfies the
inequalities
\[
M^{-1} \le a(x) \le M
\]
for all $x\in \mathbf{R}^{d}$ and a constant $M >1$;

(C2) Function $b(x,\mu)$ is continuous and bounded on $\R^d\times \MC(\R^d)$, $b(.,\mu)\in C^2(\R^d)$, and $b$ is Lipshitz continuous as a function of $x$, uniformly in other variables;

(C3) The first and second order variational derivatives of $b(x,\mu)$ with respect to $\mu$ are well defined, bounded and
\[
b(x,.)\in (C^{2,1\times 1}\cap C^{1,2})(\MC_{\le \la}(\R^d))
\]
for any $\la>0$.

The well posedness of this  nonlinear non-stochastic equation with random coefficients, \eqref{eqgenchar2} or \eqref{eqLdres1} follows from the general results on the well-posedness of nonlinear stable-like equations  
from \cite{Kol2000, Kol2007, Kol1, KolTro16}. Due to the equivalence of this equation with equation \eqref{eqlimMacVlaSPDEst}
we obtain the following well-posedness result for \eqref{eqlimMacVlaSPDEst}.
   
\begin{theorem}
	\label{thwellposMcVla_StL}
Under assumptions (C1)-(C3) and any given $T>0$ the following holds.

(i) The Cauchy problem for equation \eqref{eqlimMacVlaSPDEst} is well posed almost surely, that is, for any initial condition $Y\in \MC^+(\R^d)$ 
it has the unique bounded nonnegative solution $\mu_t(Y)$ such that $\zeta_t=\exp\{-\Om^* W_t\}\mu_t$ solves \eqref{eqgenchar2} and \eqref{eqLdres1}.

(ii) For all $t>0$, $\|\zeta_t\|\le \|Y\|$ and $\zeta_t$ have densities, $g_t$, with respect to Lebesgue measure, which satisfies also to the mild equation
\begin{equation}
\label{eq1athwellposMcVla2multi}
g_t(x)=\int G_t(x,y) Y(y)(dy)	- \int_0^t ds \int \frac{\pa}{\pa y}\left( G_{(t-s)} (x,y) \tilde b(W_s,y, \zeta_s)\right) g_s (y)\,  dy.
\end{equation} 
Consequently, $\|\mu_t\|\le \|Y\|$ and $\mu_t$ also have densities, $v_t$ 
and $\mu_t(dy)=v_t(y)dy \to Y$ weakly, as $t \to 0$. If the initial condition $Y$ has a density, $v_0$, then $v_t\to v_0$ in $L^1(\R^d)$.

(iii)	For any two solutions $\mu_t^1$ and $\mu_t^2$ of \eqref{eqlimMacVlaSPDEst}
with the initial conditions $Y^1,Y^2$ the estimate
\begin{equation}
\label{eq2thwellposMcVlaa}
\|\mu_t^1 -\mu_t^2\|_{L_1(R^d)} \le \|Y^1-Y^2\|_{\MC(R^d)}C(T,\|Y^1\|)
\end{equation}
with $C$ depending on the bounds of the derivatives in conditions (C1)-(C3).
\end{theorem}

Let us stress that the coefficients of equation  \eqref{eqlimMacVlaSPDEst} are random, but the estimates of growth \eqref{eq2thwellposMcVlaa} 
are deterministic, because of the uniform bounds for coefficients $\tilde b$ and $\tilde a$. In fact, 
the constants $C(T,\|Y^1\|)$ on the r.h.s. of \eqref{eq2thwellposMcVlaa} can be naturally expressed in terms of the Mittag-Leffler functions, see \cite{Kol18}.  

\section{Sensitivity: first order}

In this section we shall study the sensitivity for the nonlinear stochastic McKean-Vlasov equations \eqref{eqlimMacVlaSPDEst}, that is, 
the derivatives 
\begin{equation}
	\label{eqvarderdef}
	\xi_t(x;.)=\xi_t(x;.)[\mu_0]=\frac{\de \mu_t}{\de \mu_0(x)}=\frac{d}{dh}|_{h=0} \mu_t[\mu_0+h\de_x],
\end{equation}
and
\[
\eta_t(x,z;.)=\eta_t(x,z;.) [\mu_0]=\de^2\mu_t/\de \mu_0(x)\de \mu_0(z). %, \quad \tilde %\eta_t(x,z;.)=\de^2v_t(Y)/\de Y(x)\de Y(z).
\]

By Lemma  \ref{stochcharb} these derivatives are expressed in terms of the derivatives of the non-stochastic equation    
\eqref{eqgenchar2} or \eqref{eqLdres1}. For nonstochastic equations of this kind the sensitivity was obtained in our previous paper \cite{KolTro16},
which implies the point-wise (for almost all trajectories $W_t$) sensitivity of   \eqref{eqlimMacVlaSPDEst}.
However, what is important for applications is to have estimates of growth of the derivatives $\xi, \eta$ in various functional spaces
that are either deterministic or at least bounded in expectation. For the case of constant correlations $\sigma_{com}$ we shall be able to get 
bounds that are deterministic. To this end, we have to look at the main stages of the proof of the sensitivity of           
\eqref{eqgenchar2} and \eqref{eqLdres1}  and to see exactly how the estimates depend on the coefficients. 
The key point to note is that, since $\tilde b$ and $\tilde a$ are obtained by the shifting of $b$ and $a$, the assumptions (C1)-(C3)
on $b$ and $a$ are equivalent to the same assumptions on $\tilde b$ and $\tilde a$ with the same bounds on the norms in all spaces involved.

\begin{theorem}
\label{thdervarder2multiSens}
 Under Conditions (C1)--(C3) the mapping $Y=\mu_0 \mapsto \mu_t$ from Theorem \ref{thwellposMcVla_StL} is continuously differentiable in $Y$, so that 
the derivative \eqref{eqvarderdef} is well defined as a continuous function of two variables such that 
\begin{equation}
\label{eq1thdervarder2multiSens}
\sup_x \sup_{\mu_0\in \MC^+_{\le \la}(\R^d)} \|\xi_t(x;.)[\mu_0]\| \le C(T,\la, \|Y\|)
\end{equation}   
for $t\in [0,T]$ uniformly for all values of $W_t$. Moreover, the variational derivative $\xi_t(x;.)$
are twice differentiable in $x$ weakly, as the functionals on the spaces of smooth functions, that is,
\[
\pa \xi_t(x;.)/\pa x\in (C^1_{\infty}(\R^d))^*, \quad
\pa^2 \xi_t(x;.)/\pa x^2 \in (C^2_{\infty}(\R^d))^*,
\]
and
\[
\pa \xi_t(x;.)/\pa x|_{(C^1_{\infty}(\R^d))^*} \le  C(T,\la, \|Y\|),
\]
\begin{equation}
\label{eq2thdervarder2multiSens}
\pa \xi_t(x;.)/\pa x|_{(C^2_{\infty}(\R^d))^*} \le  C(T,\la, \|Y\|),
\end{equation}
again uniformly for all values of the noise $W_t$.
\end{theorem}
      
\begin{proof} By Lemma \ref{stochcharb} it is sufficient to prove the result for the mapping 
$Y=\mu_0 \to \zeta_t=\exp\{-\Om^* W_t\}\mu_t$ solving \eqref{eqgenchar2} or \eqref{eqLdres1}.
As mentioned above, since this equation is not stochastic we can apply the results of paper \cite{KolTro16}
to derive that all required derivatives are well defined and moreover, $\tilde \xi_t=\de \zeta_t/\de \mu_0(x)$ itself 
and its derivatives $\pa \tilde \xi_t(x;.)/\pa x$, $\pa^2 \tilde \xi_t(x;.)/\pa x^2$
solve the weak equation
\[
\frac{d}{dt}(f,\tilde \xi_t(x;.))=\left((\tilde b(W_t,. \, ,\zeta_t),\nabla)f -\tilde{a}(W_t,.)|\Delta|^{\alpha/2}f, \tilde \xi_t(x;.)\right)
\]
\begin{equation}
\label{eq3thdervarder2multiSens}
+\int \int \left( \frac{\de \tilde b(W_t, y, \zeta_t)}{\de \zeta_t(w)}\tilde \xi_t(x,w), \nabla f(y) \right) \zeta_t(y)\, dy dw,
\end{equation}
 obtained by the formal differentiation of equation \eqref{eqgenchar2}, \eqref{eqLdres1}. The solution to this equation is obtained by 
 considering the last term as the perturbation to the equation defined by the first term, which in its turn is solved by duality
 from the dual backward equation (in backward time) 
\begin{equation}
\label{eq4thdervarder2multiSens}
\frac{d}{dt}f=-(\tilde b(W_t,. \, ,\zeta_t),\nabla)f +\tilde{a}(W_t,.)|\Delta|^{\alpha/2}f.
\end{equation}            
We can see now that, due to the uniformity (with respect to the noise) of the norms in (C1)-(C3), all bounds for the solutions to 
equation \eqref{eq3thdervarder2multiSens} obtained in \cite{KolTro16} are also uniform completing the proof.  
\end{proof}

\section{Sensitivity: second order}

Similarly to Theorem \ref{thdervarder2multiSens} we can now derive the following result on the second order sensitivity for the McKean-Vlasov equation  (\ref{eqlimMacVlaSPDEst}).

We obtain for $\tilde\eta_t=\de^2 \zeta_t/\de Y(x)\de Y(z)$
the weak equation by differentiating equation  (\ref{eq3thdervarder2multiSens}) 
\[
\frac{d}{dt}(f,\tilde\eta_t(x,z;.))
=\left((\tilde b(W_t, . \, ,\zeta_t),\nabla)f -\tilde{a}(W_t,.)|\Delta|^{\alpha/2}f, \tilde\eta_t(x,z;.)\right)+(f,q_t)
\]
\begin{equation}
\label{McVlassecderSt}
+\int \int \left( \frac{\de \tilde b_{}(W_t, y, \zeta_t)}{\de \zeta_t(w)}\tilde\eta_t(x,z;w), \nabla f(y) \right) \zeta_t(y)\, dy dw
\end{equation}
with $(f,q_t)$ being given by
\[
\int \int \left( \frac{\de \tilde b(W_t, y, \zeta_t)}{\de \zeta_t(w)}, \nabla f(y) \right)
[\tilde\xi_t(x;y)\tilde\xi_t(z;w)+ \tilde\xi_t(x;w)\tilde\xi_t(z;y)] \, dy dw
\]
\begin{equation}
\label{McVlassecder1St}
+ \int \int \iint
\left( \frac{\de ^2 \tilde b(W_t, y, \zeta_t)}{\de \zeta_t(w)\de \zeta_t(u)}, \nabla f(y) \right)
\tilde\xi_t(x;w)\tilde\xi_t(z;u) \zeta_t(y) \, dy dw du,
\end{equation}
which should be satisfied with the vanishing initial condition.

This is the same equation as (\ref{eq3thdervarder2multiSens}), but with the additional non-homogeneous term $(f, q_t)$. 

The structure of (\ref{McVlassecder1St}) conveys an important message that for this analysis one needs the exotic spaces
$C^{2,k \times k} (\MC^+_{\le \la}(\R^d))$ introduced in the introduction.

\begin{theorem}
	\label{propdervardersecSt}
	(i) Under Conditions (C1)--(C3) the mapping $Y=\mu_0 \mapsto \mu_t$ from Theorem \ref{thwellposMcVla_StL} is twice continuously differentiable in $Y$, so that the derivative $\eta_t(x,z;.)$ is well defined as a continuous function of three variables such that 
	%Then $\eta_t(x,z;.)$ is well-defined for any $t$ as an element of $(C^1(\R^d))^*$ and it has the following bound:
\begin{equation}
\label{eq1thdervarder2multiSens}
\sup_x \sup_z \sup_{\mu_0\in \MC^+_{\le \la}(\R^d)} \|\eta_t(x,z;.)[\mu_0]\| \le C(T,\la, \|Y\|)
\end{equation}   
for $t\in [0,T]$ uniformly for all values of $W_t$. 

Moreover, the derivatives of $\eta_t(x,z;.)$ with respect to $x$ and $z$ of order at most one are well-defined as elements of $(C^2(\R^d))^*$ and
	\begin{equation}
		\label{eq4propdervardersec}
			\|\frac{\pa^{\ga}}{\pa x^{\ga}} \frac{\pa^{\be}}{\pa z^{\be}} \eta_t(x,z;.)\|_{(C^2(\R^d))^*}
		\le   C(T,\la, \|Y\|)
	\end{equation}
	for $\ga,\be=0,1$.
\end{theorem}

\begin{proof} The proof is analogous to that of Theorem \ref{thdervarder2multiSens} and is based on the application of the results of paper \cite{KolTro16}.
%	
%		By Lemma \ref{stochcharb} it is sufficient to prove the result for the mapping $Y=\mu_0 \to \zeta_t=\exp\{-\Om^* W_t\}\mu_t$ solving \eqref{eqgenchar2} and \eqref{eqLdres1}.
%	As mentioned above, since this equation is not stochastic we can apply the results paper \cite{KolTro16}
%	to derive that all required derivatives are well defined and moreover, 
%	$\tilde\eta_t=\de^2 \zeta_t/\de Y(x)\de Y(z)$
%	itself and its derivatives solve the weak equation (\ref{McVlassecderSt}). 
%
\end{proof}
% Acknowledgement
\section{ACKNOWLEDGMENTS}
The second author's work has been supported by the Ministry of Education and Science of the Russian Federation (Grant No. 1.6069.2017/8.9).

\end{document}